\documentclass[a4paper,11pt]{amsart}
\usepackage{amsmath,amsfonts,amssymb,amsthm,enumerate}
\usepackage[left=2.5cm,right=2.5cm,top=3cm,bottom=3cm,a4paper]{geometry}
\usepackage[onehalfspacing]{setspace}
\usepackage{mathrsfs}
\usepackage[pdftex]{graphicx}
\usepackage[all]{xy}
\usepackage{tikz}
\usepackage{float}
\usepackage{tikz-3dplot}

\usetikzlibrary{positioning}
\tikzset{>=stealth}

\theoremstyle{plain}
\newtheorem{thm}{Theorem}[section]
\newtheorem{lem}[thm]{Lemma}
\newtheorem{prop}[thm]{Proposition}
\newtheorem{cor}[thm]{Corollary}

\theoremstyle{definition}
\newtheorem{defn}[thm]{Definition}

\newtheorem{rmk}[thm]{Remark}
\newtheorem{note}[thm]{Note}

\newcommand{\OO}{\mathcal{O}}

\title{Higher syzygies on abelian surfaces}
\author{Jaesun Shin}
\date{}

\address{Department of Mathematical Sciences, KAIST, 291 Daehak-ro, Yuseong-gu, Daejon 305-701, Korea}

\email{jsshin1991@kaist.ac.kr}

\begin{document}
\maketitle
\begin{abstract}
Based on the theory of an infinitesimal Newton-Okounkov body, we extend the results of Lazarsfeld-Pareschi-Popa \cite{LPP} on abelian surfaces. Moreover, we show that the higher syzygies of $(X,L)$ are completely determined by its Seshadri constant when $L^{2}$ is large. As an application, we improve the existing lower bound of $(L^{2})$ for higher syzygies of a polarized abelian surface $(X,L)$. 
\end{abstract}

\begin{section} {Introduction}
The study of how a variety can be embedded in a projective space is an important subject in algebraic geometry. The natural way to embed a variety $X$ to a projective space $\mathbb{P}$ is to consider a very ample line bundle $L$ on $X$. Once we know that $L$ is very ample, the next step is then to study its image on a projective space. This is highly related to the section ring of $L$ and its syzygy modules. This is why we study the higher syzygies of $L$ on $X$. 

The algebraic properties of $R(X,L)$ have a significant meaning in algebraic geometry since they imply many geometric properties of $X$. From Castelnuovo and Mumford (\cite{M}) to Green and Lazarsfeld (\cite{GL}), a new perspective was provided to this problem by studying the minimal graded resolution of $R(X,L)=\bigoplus H^{0}(X, L^{\otimes d})$. More precisely, let $S={\rm Sym}^{\bullet}(H^{0}(X,L))$ be the homogeneous coordinate ring of the projective space $\mathbb{P}$, and consider the graded $S$-module $R(X,L)$. As an $S$-module, the minimal graded free resolution of $R(X,L)$ looks like 
\begin{align*}
\cdots \rightarrow E_{p} \rightarrow \cdots \rightarrow E_{1} \rightarrow E_{0} \rightarrow R(X,L) \rightarrow 0,
\end{align*}  
where $E_{0}=S \oplus \bigoplus_{j}S(-a_{0j})$, $E_{1}=\bigoplus_{j} S(-a_{1j})$, and in general $E_{p}=\bigoplus_{j}S(-a_{pj})$ with $a_{pj} \ge p+1$ for any $j$. Then $L$ is said to satisfy property $N_{0}$ if $E_{0}=S$. Moreover, $L$ is said to satisfy property $N_{p}$ if it satisfies property $N_{p-1}$ and $a_{pj}=p+1$ for any $j$. 

As an illustration, translated to geometric terms, property $N_{0}$ means that the map $S \rightarrow R(X,L)$ is surjective, which implies that the Kodaira map $\phi_{L}$ induces a projectively normal embedding of $X$ into $\mathbb{P}$. Also, property $N_{1}$ is equivalent to further requiring that the homogeneous ideal of $X$ be generated by quadrics. 

Due to its geometric importance, there are many results that ensure property $N_{p}$ for $L$. From the work of Castelnuovo, Mattuck, Fujita, and Saint-Donat, Green (\cite{MG}) proved that $L$ satisfies property $N_{p}$ if ${\rm deg}L \ge 2g(X)+p+1$ when $X$ is a curve. This result stimulated many interesting questions. (See \cite{EL, GP, P} for instance.) One of these is how to connect some numerical invariants to property $N_{p}$ for $L$. In this regard, Lazarsfeld-Pareschi-Popa (\cite{LPP}) showed that if $\epsilon(X,L)>(p+2)g$, where $(X,L)$ is a polarized abelian variety of dimension $g$ and $\epsilon(X,L)$ is its Seshadri constant, and then $L$ satisfies property $N_{p}$. This extends the result of Hwang and To (\cite{HT}) on projective normality to higher syzygies.  

The main purpose of this paper is to analyze higher syzygies of a polarized abelian surface $(X,L)$. Since ${\rm dim}X=2$, \cite[Theorem A]{LPP} can be rephrased as follows: if $\epsilon(X,L)>2(p+2)$, then $L$ satisfies property $N_{p}$. In contrast to \cite{P}, the interesting aspect of statements of this kind lies in the case when $L$ is primitive: it is the first statement for higher syzygies of primitive line bundles. It is then natural to ask about higher syzygies of $L$ when $\epsilon(X,L) \le 2(p+2)$. Since \cite[Theorem A]{LPP} uses the assumption $\epsilon(X,L)>2(p+2)$ to construct a divisor whose multiplier ideal sheaf is `nice' (cf. \cite[Lemma 1.2]{LPP}), it is difficult to answer this question using the techniques of \cite{LPP}. 

By adjusting $L^{2}$ and $\epsilon(X,L)$, we provide an answer to this question. As a consequence of our result, we show property $N_{p}$ for $L$ when $\epsilon(X,L) \ge 2(p+2)$ (cf. Corollary \ref{cor:generalization LPP}). In this point of view, Theorem \ref{thm:main} is a generalization of \cite[Theorem A]{LPP} on abelian surfaces. In terms of applications, our result (Theorem \ref{thm:introduction}) is much more flexible than \cite[Theorem A]{LPP} by controlling the Seshadri constant and the self-intersection number. (See Corollary \ref{cor:general criterion} and Remark \ref{rmk:comparison with LPP}.) 

Our main result is the following. 

\begin{thm} {\rm (=Theorem \ref{thm:main})} \label{thm:introduction}
Let $(X,L)$ be a polarized abelian surface. Assume that 
\begin{align*}
(L^{2}) \cdot (\epsilon(X,L)-p-2)-(p+2) \cdot \epsilon(X,L)^{2} > 0.
\end{align*}
Then $L$ satisfies property $N_{p}$. 
\end{thm}

Our approach in proving Theorem \ref{thm:introduction} relies on the theory of infinitesimal Newton-Okounkov bodies. The first work confirming property $N_{p}$ for a polarized abelian surface by using infinitesimal Newton-Okounkov bodies was introduced by K\" uronya and Lozovanu (\cite{KL1509}). We use their results and the important features of generic infinitesimal Newton-Okounkov bodies (\cite[Proposition 4.2]{KL1411}) to prove Theorem \ref{thm:introduction}.  

Furthermore, we show that the higher syzygies of $(X,L)$ are completely determined by $\epsilon(X,L)$ when $L^{2}$ is large. 

\begin{thm} {\rm (=Theorem \ref{thm:large})}
Let $(X,L)$ be a polarized abelian surface. Assume that $L^{2}>(p+2)(p+3)^{2}$. Then the following are equivalent:
\begin{enumerate}[(1)]
\item $\epsilon(X,L)>\frac{(L^{2})-\sqrt{{(L^{2})}^{2}-4(p+2)^{2}(L^{2})}}{2(p+2)}$. 
\item $L$ satisfies property $N_{p}$. 
\end{enumerate}
\end{thm}

Turning to applications, it is interesting to study Theorem \ref{thm:introduction} by using the bounds of the Seshadri constant, as in \cite[Corollary B]{LPP}. In general, it is difficult to control the Seshadri constant. However, it was shown in \cite{L96} that on a polarized abelian variety $(A,L)$ of dimension $g$, the Seshadri constant of $L$ can be estimated by a metric invariant, which is called the minimal period length of $(A,L)$. By adapting this argument, Bauer (\cite{T97}), Di Rocco, Harbourne, Kapustka, Knutsen, Syzdek, and Szemberg (\cite{TSBMAWT}) showed that if $(A,L)$ is very general, then
\begin{align*}
\frac{\sqrt[g]{2 \cdot (L^{g})}}{4} \le \epsilon(A,L) \le \sqrt[g]{(L^{g})}.
\end{align*}
In particular, it can be written as $\frac{1}{2\sqrt{2}}\sqrt{(L^{2})} \le \epsilon(A,L) \le \sqrt{(L^{2})}$ when ${\rm dim}A=2$. 

Let $(X,L)$ be a very general polarized abelian surface. By \cite[Theorem A]{LPP} and \cite[Theorem 1.(b)]{T97}, it is immediate that $L$ satisfies property $N_{p}$ if $L^{2} >32(p+2)^{2}$ (\cite[Corollary B]{LPP}). However, we prove that such a lower bound of $(L^{2})$ can be reduced to $\frac{81}{8}(p+2)^{2}$.

Moreover, when $(X,L)$ is a polarized abelian surface satisfying $\epsilon(X,L) \notin \mathbb{Z}$, we show that $\epsilon(X,L) \ge \sqrt{\frac{1}{2}(L^{2})}$ by using the theory of infinitesimal Newton-Okounkov bodies. Therefore we obtain:

\begin{cor} {\rm (=Corollary \ref{cor:general criterion})} 
Let $(X,L)$ be a polarized abelian surface, and let $p \ge 0$ be an integer. 
\begin{enumerate}[(1)]
\item Assume that $(X,L)$ is very general. If $L^{2}>\frac{81}{8}(p+2)^{2}$, then $L$ satisfies property $N_{p}$. In particular, the converse holds if $L^{2}>(p+2)(p+3)^{2}$. 
\item Assume that $\epsilon(X,L) \notin \mathbb{Z}$. If $L^{2}>\frac{9}{2}(p+2)^{2}$, then $L$ satisfies property $N_{p}$.
\end{enumerate}
\end{cor}

\textbf{Notation and convention.} 
In this paper, we work over the complex number field and a divisor means an integral Cartier divisor. For an effective $\mathbb{Q}$-divisor $D$ on a smooth projective variety $X$, the notation $\mathcal{J}(X,D)$ stands for the multiplier ideal sheaf associated to $D$ (cf. \cite[Definition 9.2.1]{L}). \\

\textbf{Organization of the paper.}
In Section \ref{section:preliminary}, we recall some results and techniques of infinitesimal Newton-Okounkov bodies that are useful in approaching higher syzygies on polarized abelian surfaces. Section \ref{section:main} is devoted to proving Theorem \ref{thm:main} and Theorem \ref{thm:large}. We also check the higher embeddings and the Koszul property of polarized abelian surfaces. Section \ref{section:application} deals with many applications of the results in Section \ref{section:main}. \\

\textbf{Acknowledgements.}
I would like to thank my advisor Yongnam Lee, for his advice, encouragement and teaching. This work was supported by NRF(National Research Foundation of Korea) Grant funded by the Korean Government(NRF-2016-Fostering Core Leaders of the Future Basic Science Program/Global Ph.D. Fellowship Program) and was partially supported by the grant 346300 for IMPAN from the Simons Foundation and the matching 2015-2019 Polish MNiSW fund. 
\end{section}

\begin{section} {Higher syzygies and infinitesimal Newton-Okounkov bodies} \label{section:preliminary}

Throughout this section, $X$ is a smooth projective variety of dimension $n$, and $x \in X$ is a point. Let $\pi:{\rm{Bl}}_{x}(X)=X' \rightarrow X$ be the blow-up of $X$ at $x$ with the exceptional divisor $E$. 

\begin{subsection} {Infinitesimal Newton-Okounkov bodies}
See \cite{LM} for definition and basic properties of Newton-Okounkov bodies. We start by defining the infinitesimal Newton-Okounkov bodies. 

\begin{defn}
Let $L$ be a big divisor on $X$. The infinitesimal Newton-Okounkov body of $L$ over $x$ is defined to be the Newton-Okounkov body $\Delta_{X'_{\bullet}}(\pi^{*}L)$ associated to a flag $X'_{\bullet}: X' \supseteq E \supseteq X'_{2} \supseteq \cdots \supseteq X'_{n}=\{z\}$, where $X'_{i} \cong \mathbb{P}^{n-i}$ is a linear subspace of $E \cong \mathbb{P}^{n-1}$. Furthermore, if $x$ and $X'_{\bullet}$ are chosen to be very general, we call it the generic infinitesimal Newton-Okounkov body. We denote it by $\Delta_{x}(L)$. 
\end{defn}

\begin{rmk} 
\begin{enumerate}[(1)]
\item By \cite[Proposition 5.3]{LM}, $\Delta_{x}(L)$ is well-defined. 
\item For notational convenience, when $X$ is a surface and $L$ is a big line bundle on $X$, we denote the infinitesimal Newton-Okounkov body of $L$ associated to a flag $X'_{\bullet}:X' \supseteq E \supseteq \{z\}$ by $\Delta_{(E,z)}(\pi^{*}L)$. 
\end{enumerate} 
\end{rmk}

We recall the inverted standard simplex and the largest inverted simplex constant of a divisor on a surface which play a vital role in local positivity. Note that the following definition of the largest inverted simplex constant makes sense by \cite[Theorem 4.1]{KL1507} and \cite[Proposition 4.6]{KL1507}. For higher-dimensional one, see \cite[Defintion 2.5]{KL1507} and \cite[Definition 4.4]{KL1507}. 

\begin{defn}
The inverted standard simplex of length $\xi$ is defined to be 
\begin{align*}
\Delta_{\xi}^{-1}:=\{(t,y) \in \mathbb{R}^{2}|\text{ } 0 \le t \le \xi, 0 \le y \le t\}.
\end{align*} 
Moreover, if $L$ is an ample line bundle on a smooth projective surface $X$, the largest inverted simplex constant is then defined as  
\begin{align*}
\xi(L;x):={\rm sup}\{\xi>0| \text{ } \Delta_{\xi}^{-1} \subseteq \Delta_{(E,z)}(\pi^{*}L)\}.
\end{align*}
If $L$ is big but not ample, we may let $\xi(L;x)=0$. 
\end{defn}

\begin{note}
From now on, we denote by $\mu(L,x):={\rm sup}\{t>0 |\text{ } \pi^{*}L-tE \text{ is big}\}$, where $L$ is a big line bundle on $X$. 
\end{note}

Next, we quickly recall a few notions and useful facts without proof. 

\begin{prop} {\rm (\cite[Proposition 3.1]{KL1411})} \label{prop:base}
Let $L$ be a big line bundle on a smooth projective surface $X$. 
\begin{enumerate}[(1)]
\item $\Delta_{(E,z)}(\pi^{*}L) \subseteq \Delta_{\mu(L,x)}^{-1}$ for any $z \in E$. 
\item There exist finitely many points $z_{1}, \dots, z_{k} \in E$ such that $\Delta_{(E,z)}(\pi^{*}L)$ is independent of $z \in E-\{z_{1}, \dots, z_{k}\}$, with base the whole line segment $[0, \mu(L,x)] \times \{0\}$. 
\end{enumerate}
\end{prop}

\begin{defn}
Let $L$ be a big divisor on a smooth projective variety $X$, and let $x \in X$ be a smooth point such that $x \notin {\rm \textbf{B}_{+}}(L)$. The real number 
\begin{align*}
\epsilon(\lVert L \rVert ; x):=\sup_{f^{*}D=A+E} \epsilon(A;x),
\end{align*}
is the moving Seshadri constant of $L$ at $x$. The supremum in the definition is taken over all projective morphisms $f:Y \rightarrow X$ with $Y$ smooth and $f$ an isomorphism around $x$, and over all decompositions $f^{*}D=A+E$, where $A$ is ample and $E$ is effective with $f^{-1}(x) \notin {\rm Supp}(E)$. \

Moreover, if $x \in {\rm \textbf{B}_{+}}(L)$, we may let $\epsilon(\lVert L \rVert ; x)=0$. 
\end{defn}

\begin{prop} {\rm (\cite[Corollary 4.11]{KL1507})} \label{prop:moving Seshadri constant}
Let $L$ be a big divisor on a smooth projective variety $X$. Then
\begin{align*}
\xi(L;x)=\epsilon(\lVert L \rVert;x)
\end{align*}
for any $x \in X$, where $\xi(L;x)$ is the largest inverted simplex constant {\rm (cf. \cite[Definition 4.4]{KL1507})}.
\end{prop}

The following proposition shows that the Seshadri constant of an ample line bundle at a very general point restricts the shape of its generic infinitesimal Newton-Okounkov body roughly. 

\begin{prop} {\rm (\cite[Proposition 4.2]{KL1411})} \label{prop:Seshadri surface}
Let $L$ be an ample integral Cartier divisor on a smooth projective surface $X$, and let $x \in X$ be a very general point. Assume that $\epsilon(L;x)$ is submaximal. Then there exists a Seshadri exceptional curve $F \subset X$ with $(L.F)=p$ and ${\rm mult}_{x}(F)=q$ such that $\epsilon(X,L)=\frac{p}{q}$. Moreover, 
\begin{enumerate}[(1)]
\item if $q \ge 2$, then $\Delta_{x}(L) \subseteq \Delta_{OAB}$, where $O=(0,0)$, $A=(\frac{p}{q}, \frac{p}{q})$, and $B=(\frac{p}{q-1}, 0)$. 
\item if $q=1$, then $\Delta_{x}(L)$ is contained in the area below the line $y=t$, and between the lines $y=0$ and $y=\epsilon(X,L)$. 
\end{enumerate}
\end{prop}

\end{subsection}

\begin{subsection} {Higher syzygies using the infinitesimal Newton-Okounkov bodies}
In \cite{LPP} and \cite{KL1509}, they studied higher syzygies of a polarized abelian surface by using the infintesimal Newton-Okounkov bodies. From now on, we further assume that $X$ is an abelian surface. 

\begin{thm} {\rm (\cite{LPP} or \cite[Theorem 3.1]{KL1509})} \label{thm:LPP1} 
Let $(X,L)$ be a polarized abelian surface, and let $p \ge 0$ be an integer such that there exists an effective $\mathbb{Q}$-divisor $F_{0}$ on $X$ such that 
\begin{enumerate}[(1)]
\item $F_{0} \equiv \frac{1-c}{p+2}L$ for some $0<c<1$, and
\item $\mathcal{J}(X,F_{0})=\mathcal{I}_{0}$, the maximal ideal at the origin.
\end{enumerate}
Then $L$ satisfies property $N_{p}$.
\end{thm}

Now, let us give a sketch of the proof of Theorem \ref{thm:LPP1}. The main idea is to prove the vanishing of 
\begin{align*}
H^{1}(X^{\times (p+2)}, \overset{p+2}\boxtimes L \otimes \mathcal{I}_{\Sigma}),
\end{align*}
where $\mathcal{I}_{\Sigma}$ is the ideal sheaf of the reduced algebraic set $\Sigma=\{(x_{0}, \dots, x_{p+1}) \in X^{\times (p+2)}| \text{ } x_{0}=x_{i} \text{ for some } 1 \le i \le p+1\}=\Delta_{0,1} \cup \cdots \cup \Delta_{0,p+1}$. By \cite[Lemma 1.2]{LPP} and techniques of multiplier ideal, they showed $\mathcal{I}_{\Sigma}=\mathcal{J}(X^{p+2},E)$ for some line bundle $E$. Also, by some computations using Poincar\'e bundle, the ampleness of $\overset{p+2}\boxtimes L (-E)$ is obtained. Then the Nadel vanishing gives the desired vanishing $H^{1}(X^{\times (p+2)}, \overset{p+2}\boxtimes L \otimes \mathcal{I}_{\Sigma})=0$. Since this implies property $N_{p}$ for $L$ (cf. \cite{I}), Theorem \ref{thm:LPP1} holds. 

On an abelian surface, the main theorem of \cite{LPP} can be rephrased as follows. 

\begin{thm} {\rm (\cite[Theorem A]{LPP})} \label{thm:LPP}
Let $(X,L)$ be a polarized abelian surface. Assume that 
\begin{align*}
\epsilon(X,L)>2(p+2).
\end{align*}
Then $L$ satisfies property $N_{p}$. 
\end{thm}

Combining Theorem \ref{thm:LPP} with \cite{TSBMAWT}, Lazarsfeld-Pareschi-Popa showed:

\begin{cor} {\rm (\cite[Corollary B]{LPP})} \label{cor:Seshadri LPP}
Let $(X,L)$ be a very general polarized abelian surface of type $(d_{1},d_{2})$. Assume that $d_{1}d_{2}>16(p+2)^{2}$. Then $L$ satisfies property $N_{p}$. 
\end{cor}

\begin{thm} {\rm (\cite[Theorem 3.4]{KL1509})} \label{thm:infinitesimal Newton-Okounkov body}
Let $X$ be an abelian surface, and let $B$ be an ample $\mathbb{Q}$-divisor on $X$. Suppose that 
\begin{align*}
{\rm length}(\Delta_{(E,z)}(\pi^{*}B) \cap \{2\} \times \mathbb{R})>1,
\end{align*}
for some point $z \in E$. Then there exists an effective $\mathbb{Q}$-divisor $D \equiv (1-c)B$ for some $0<c<1$ such that $\mathcal{J}(X,D)=\mathcal{I}_{0}$ over the whole of $X$. 
\end{thm}

\begin{rmk} 
Let $p \in \mathbb{Z}_{\ge 0}$, and consider a $\mathbb{Q}$-divisor $B=\frac{1}{p+2}L$ for an ample line bundle $L$ on $X$. If such $B$ satisfies the condition on Theorem \ref{thm:infinitesimal Newton-Okounkov body}, then property $N_{p}$ holds for $L$ by Theorem \ref{thm:LPP1}. Therefore Theorem \ref{thm:infinitesimal Newton-Okounkov body} provides a combinatorial way to approach property $N_{p}$ for a given ample line bundle by using the infinitesimal Newton-Okounkov bodies. 
\end{rmk}




\end{subsection}

\begin{subsection} {Koszul property of section rings}
In \cite{LPP}, Lazarsfeld-Pareschi-Popa used the Seshadri constant to get the Koszul property of a given ample line bundle. The main ingredient is similar to Theorem \ref{thm:LPP}. We recall the definition of Koszul algebra. 

\begin{defn}
Let $R=\oplus_{i=0}^{\infty}R_{i}$ be a graded $k$-algebra with $R_{0}=k$, where $k$ is a ground field. Then $R$ is a Koszul $k$-algebra (or simply Koszul) if the trivial $R$-module $k$ admits a linear minimal graded free $R$-module resolution
\begin{align*}
\cdots \rightarrow R(-i)^{\oplus b_{i}} \rightarrow \cdots \rightarrow R(-1)^{\oplus b_{1}} \rightarrow R \rightarrow k \rightarrow 0. 
\end{align*} 
\end{defn}

\begin{prop} {\rm (\cite[Proposition 3.1]{LPP} or \cite[Proposition 3.6]{KL1509})} \label{prop:Koszul LPP} 
Let $(X,L)$ be a polarized abelian surface. Suppose that there exists an effective $\mathbb{Q}$-divisor $F_{0}$ such that 
\begin{enumerate}[(1)]
\item $F_{0} \equiv \frac{1-c}{3}L$ for some $0<c<1$, and 
\item $\mathcal{J}(X,F_{0})=\mathcal{I}_{0}$, the maximal ideal at the origin.
\end{enumerate}
Then $R(X,L)$ is Koszul. 
\end{prop}

\begin{rmk}
Note that Proposition \ref{prop:Koszul LPP} holds for abelian varieties of any dimension. 
\end{rmk}

On an abelian surface, \cite[Proposition 3.1]{LPP} can be rephrased as follows. 

\begin{prop} {\rm (\cite[Proposition 3.1]{LPP})} \label{prop:Koszul LPP2}
With notation as in Proposition \ref{prop:Koszul LPP}, assume that $\epsilon(X,L)>6$. Then $R(X,L)$ is Koszul. 
\end{prop}

\begin{cor} {\rm (\cite{LPP})} \label{cor:Koszul LPP}
Let $(X,L)$ be a very general polarized abelian surface of type $(d_{1},d_{2})$. Assume that $d_{1}d_{2}>144$. Then $R(X,L)$ is Koszul.
\end{cor}

\end{subsection}

\end{section}

\begin{section} {Higher syzygies on abelian surfaces} \label{section:main}

\begin{subsection} {Higher syzygies on $(X,L)$}
This subsection is devoted to proving Theorem \ref{thm:main}, which extends the result of Lazarsfeld-Pareschi-Popa (\cite{LPP} or Theorem \ref{thm:LPP}) in the case of abelian surfaces. Moreover, we treat its Koszulness. While Theorem \ref{thm:LPP} is proven from the computations using the Poincar\'e bundle and the techniques of multiplier ideals, our main tool is the theory of the generic infinitesimal Newton-Okounkov body. 

\begin{thm} \label{thm:main}
Let $(X,L)$ be a polarized abelian surface. Assume that 
\begin{align*}
(L^{2}) \cdot (\epsilon(X,L)-p-2)-(p+2) \cdot \epsilon(X,L)^{2} > 0.
\end{align*}
Then $L$ satisfies property $N_{p}$. 
\end{thm}

\begin{proof}
Let $B:=\frac{1}{p+2}L$ be an ample $\mathbb{Q}$-divisor on $X$, and let $x \in X$ be a very general point. Since $X$ is homogeneous, $\epsilon(L;x)=\epsilon(X,L)$. Now, we can consider the generic infinitesimal Newton-Okounkov body $\Delta_{x}(B)$. If $\epsilon(X,L) \le p+2$, then 
\begin{align*}
(L^{2}) \cdot (\epsilon(X,L)-p-2)-(p+2) \cdot \epsilon(X,L)^{2} \le -(p+2) \cdot \epsilon(X,L)^{2} <0,
\end{align*}
which contradicts the assumption. So $\epsilon(X,L) > p+2$, i.e. $\epsilon(B;x)=\epsilon(X,B)>1$. Let $l:={\rm length}(\Delta_{x}(B) \cap \{2\} \times \mathbb{R})$. We claim that $l >1$. For a contradiction, suppose that $l \le 1$. If $\epsilon(X,L) \ge 2(p+2)$, then $l=2>1$. So assume that $p+2<\epsilon(X,L)<2(p+2)$, i.e. $1<\epsilon(B;x)<2$. Let us define two kinds of convex bodies: $\Delta_{\alpha}$ for $1<\epsilon(B;x)<\alpha<2$ and $\Delta$. (See FIGURE 1 for an illustration.) 

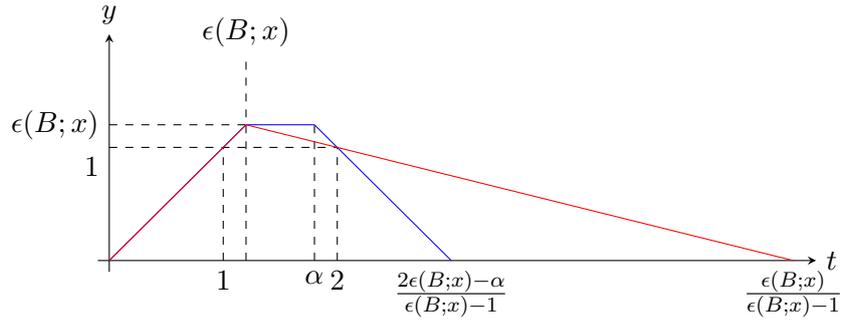
\begin{figure}[t]
\centering
\begin{tikzpicture} [scale=1.5]
      \draw[->] (-0.1,0) -- (6.2,0) node[right] {$t$};
      \draw[->] (0,-0.1) -- (0,2) node[above] {$y$};
      \node (A) at (1,0) [below] {$1$};
      \draw [dashed] (1,0) -- (1,1);
      \node (B) at (2,0) [below] {$2$};
      \draw [dashed] (2,0) -- (2,1);
      \node (C) at (1.2,1.8) [above] {$\epsilon(B;x)$};
      \draw [dashed] (1.2,0) -- (1.2,1.8);    
      \node (D) at (0,1.2) [left] {$\epsilon(B;x)$};
      \draw [dashed] (0,1.2) -- (1.2,1.2);
      \node (E) at (3,0) [below] {$\frac{2\epsilon(B;x)-\alpha}{\epsilon(B;x)-1}$};
      \node (F) at (6,0) [below] {$\frac{\epsilon(B;x)}{\epsilon(B;x)-1}$};
      \node (G) at (1.8,0) [below] {$\alpha$};
      \draw [dashed] (1.8,0) -- (1.8,1.2);   
      \node (H) at (0,1) [below left] {$1$};
      \draw [dashed] (0,1) -- (2,1); 
      
      \draw[scale=1,domain=0:1.2,smooth,variable=\t,blue] plot ({\t},{\t});
      \draw[scale=1,domain=1.2:1.8,smooth,variable=\t,blue] plot ({\t},{1.2});
      \draw[scale=1,domain=1.8:3,smooth,variable=\t,blue] plot ({\t},{-(\t-2)+1});
      
      \draw[scale=1,domain=0:1.2,smooth,variable=\t,red] plot ({\t},{\t});
      \draw[scale=1,domain=1.2:6,smooth,variable=\t,red] plot ({\t},{-0.25*(\t-2)+1});
\end{tikzpicture}
\caption{Upper boundaries of $\Delta_{\alpha}$ (blue) and $\Delta$ (red)}
\end{figure}

First, $\Delta_{\alpha}$ is a collection of $(t,y) \in \mathbb{R}^{2}$ satisfying
\begin{displaymath}
0 \le y \le \left\{ \begin{array}{ll}
t & \textrm{if $0 \le t \le \epsilon(B;x)$}\\
\epsilon(B;x) & \textrm{if $\epsilon(B;x) < t \le \alpha$}\\
\frac{\epsilon(B;x)-1}{\alpha-2}(x-2)+1 & \textrm{if $\alpha < t \le \frac{2\epsilon(B;x)-\alpha}{\epsilon(B;x)-1}$}.
\end{array} \right.
\end{displaymath}
Next, $\Delta$ is a collection of $(t,y) \in \mathbb{R}^{2}$ satisfying
\begin{displaymath}
0 \le y \le \left\{ \begin{array}{ll}
t & \textrm{if $0 \le t \le \epsilon(B;x)$}\\
\frac{\epsilon(B;x)-1}{\epsilon(B;x)-2}(x-2)+1 & \textrm{if $\epsilon(B;x) < t \le \frac{\epsilon(B;x)}{\epsilon(B;x)-1}$}.
\end{array} \right.
\end{displaymath}
Since $\Delta_{x}(B)$ is generic, by Proposition \ref{prop:Seshadri surface}, Proposition \ref{prop:moving Seshadri constant} and the assumption that $l \le 1$, $\Delta_{x}(B)$ is contained in either $\Delta_{\alpha}$ for some $1<\epsilon(B;x)<\alpha<2$ or $\Delta$. So 
\begin{align*}
{\rm vol}_{{\mathbb{R}^{2}}}(\Delta_{x}(B)) \le \max_{1<\epsilon(B;x)<\alpha<2} \{{\rm vol}_{{\mathbb{R}^{2}}}(\Delta_{\alpha}), {\rm vol}_{{\mathbb{R}^{2}}}(\Delta)\}.
\end{align*}
First, for any $1<\epsilon(B;x)<\alpha<2$, 
\begin{align*}
{\rm vol}_{{\mathbb{R}^{2}}}(\Delta_{\alpha})=-\frac{1}{2}\epsilon(B;x)^{2}+\alpha \cdot \epsilon(B;x)+\frac{2\epsilon(B;x)-\alpha \cdot \epsilon(B;x)}{2(\epsilon(B;x)-1)}\cdot \epsilon(B;x). 
\end{align*}
Next, ${\rm vol}_{{\mathbb{R}^{2}}}(\Delta)=\frac{\epsilon(B;x)^{2}}{2(\epsilon(B;x)-1)}$. Then 
\begin{align*}
\frac{2(\epsilon(B;x)-1)}{\epsilon(B;x)}\cdot ({\rm vol}_{{\mathbb{R}^{2}}}(\Delta)-{\rm vol}_{{\mathbb{R}^{2}}}(\Delta_{\alpha}))=(2-\epsilon(B;x))(\alpha-\epsilon(B;x))>0,
\end{align*}
i.e. ${\rm vol}_{{\mathbb{R}^{2}}}(\Delta)>{\rm vol}_{{\mathbb{R}^{2}}}(\Delta_{\alpha})$ for any $1<\epsilon(B;x)<\alpha<2$. Therefore, 
\begin{align*}
\frac{\epsilon(X,L)^{2}}{(p+2) \cdot 2(\epsilon(X,L)-p-2)}=\frac{\epsilon(B;x)^{2}}{2(\epsilon(B;x)-1)}&={\rm vol}_{\mathbb{R}^{2}}(\Delta) \\
&\ge {\rm vol}_{\mathbb{R}^{2}}(\Delta_{x}(B))=\frac{(L^{2})}{2(p+2)^{2}},
\end{align*}
i.e. $(L^{2}) \cdot (\epsilon(X,L)-p-2)-(p+2) \cdot \epsilon(X,L)^{2} \le 0$, which contradicts our assumption. So $l > 1$. By Theorem \ref{thm:infinitesimal Newton-Okounkov body} and Theorem \ref{thm:LPP1}, $L$ satisfies property $N_{p}$. 
\end{proof}

The proof of Theorem \ref{thm:main} implies that we do not need any condition on $(L^{2})$ if $\epsilon(X,L) \ge 2(p+2)$. So we get the following immediate corollary. This implies that Theorem \ref{thm:main} is a generalization of Theorem \ref{thm:LPP}. 

\begin{cor} \label{cor:generalization LPP}
Let $(X,L)$ be a polarized abelian surface. Assume that 
\begin{align*}
\epsilon(X,L) \ge 2(p+2).
\end{align*}
Then $L$ satisfies property $N_{p}$. 
\end{cor}

\begin{rmk} (Another proof of Theorem \ref{thm:LPP})
We can also prove Theorem \ref{thm:LPP} by using Theorem \ref{thm:main} directly. Suppose that $\epsilon(X,L)>2(p+2)$. By Proposition \ref{prop:moving Seshadri constant}, it is easy to get that $\epsilon(X,L) \le \sqrt{(L^{2})}$. Then 
\begin{align*}
(L^{2}) \cdot (\epsilon(X,L)-p-2)-(p+2) \cdot \epsilon(X,L)^{2} \ge (L^{2}) \cdot (\epsilon(X,L)-2(p+2))>0
\end{align*}
since $L$ is ample and $\epsilon(X,L)>2(p+2)$. By Theorem \ref{thm:main}, $L$ satisfies property $N_{p}$. Thus Theorem \ref{thm:main} recovers Theorem \ref{thm:LPP} as a particular case. 
\end{rmk}

Recall that a line bundle $L$ on a projective variety $X$ is called $k$-very ample if for any zero-dimensional subscheme $Z$ of $X$ of length $k+1$, the restriction map 
\begin{align*}
H^{0}(X,L) \rightarrow H^{0}(Z, {L|}_{Z})
\end{align*}
is surjective. In particular, $0$-very ample is equivalent to base point freeness of $L$ and $1$-very ample means that $L$ is very ample. This is another notion of strong positivity of a polarized variety. Moreover, it is well-known that property $N_{p}$ implies $(p+1)$-very ampleness (\cite[Remark 3.9]{EGHP}). Therefore we obtain:

\begin{cor} \label{cor:p-very ampleness}
Let $(X,L)$ be a polarized abelian surface. Assume that 
\begin{align*}
(L^{2}) \cdot (\epsilon(X,L)-p-1)-(p+1) \cdot \epsilon(X,L)^{2} > 0.
\end{align*}
Then $L$ is $p$-very ample. 
\end{cor}

\begin{proof}
When $p \ge 1$, the result is a direct consequence of Theorem \ref{thm:main}, so we only need to check it for $p=0$. In this case, our assumption implies that $L^{2} > 4$, that is, $L^{2} \ge 6$. Suppose that $L$ is not globally generated. \cite[Theorem 1.1]{T} implies that there exists an elliptic curve $C$ on $X$ satisfying $(L.C)=1$. Then the existence of such an elliptic curve gives that $(X,L)$ is a product of elliptic curves such that $\epsilon(X,L)=1$ by \cite[Lemma 2.6]{N} and \cite[Theorem 2.2.2]{TSBMAWT}. However, it contradicts our assumption $(L^{2}) \cdot (\epsilon(X,L)-1)-(p+1) \cdot \epsilon(X,L)^{2}>0$. Hence $L$ is globally generated as desired. 
\end{proof}

Another application of this idea concerns Koszul rings so that we extend Proposition \ref{prop:Koszul LPP2}. Since the proof is similar, we omit it here. 

\begin{prop} \label{prop:Koszul}
Let $(X,L)$ be a polarized abelian surface. Assume that 
\begin{align*}
(L^{2}) \cdot (\epsilon(X,L)-3)-3 \cdot \epsilon(X,L)^{2}>0.
\end{align*}
Then $R(X,L)$ is a Koszul algebra. 
\end{prop}

\begin{cor} \label{cor:Koszul}
Let $(X,L)$ be a polarized abelian surface that satisfies $\epsilon(X,L) \ge 6$. Then $R(X,L)$ is a Koszul algebra. 
\end{cor}

\end{subsection}

\begin{subsection} {Higher syzygies on $(X,L)$ for large $L^{2}$} \label{subsection:large}

We prove that the higher syzygies and the higher order embeddings of $(X,L)$ are completely determined by the lower bound of $\epsilon(X,L)$ when $L^{2}$ is large. (See \cite[Theorem 1.1]{KL1509} for the other numerical criterion.)

\begin{thm} \label{thm:large}
Let $(X,L)$ be a polarized abelian surface. Assume that $L^{2}>(p+2)(p+3)^{2}$. Then the following are equivalent:
\begin{enumerate}[(1)]
\item $\epsilon(X,L)>\frac{(L^{2})-\sqrt{{(L^{2})}^{2}-4(p+2)^{2}(L^{2})}}{2(p+2)}$. 
\item $L$ satisfies property $N_{p}$. 
\end{enumerate}
\end{thm}

\begin{proof}
The implication $(1) \Rightarrow (2)$ follows from the inequality $\sqrt{L^{2}}<\frac{(L^{2})+\sqrt{{(L^{2})}^{2}-4(p+2)^{2}(L^{2})}}{2(p+2)}$ and Theorem \ref{thm:main}, so we are left checking $(2) \Rightarrow (1)$. Since $L^{2} >(p+2)(p+3)^{2} \ge 4(p+2)^{2}$, $\frac{(L^{2})-\sqrt{{(L^{2})}^{2}-4(p+2)^{2}(L^{2})}}{2(p+2)}$ makes sense. By \cite{BS1}, it is sufficient to show the following two inequalities:
\begin{align*}
\frac{(L^{2})-\sqrt{{(L^{2})}^{2}-4(p+2)^{2}(L^{2})}}{2(p+2)}< \min\{\frac{\sqrt{7(L^{2})}}{2\sqrt{2}}, \text{ }\epsilon_{0}(L)\},
\end{align*}
where $\epsilon_{0}(L)$ is the minimal degree of an elliptic curve in $X$ with respect to $L$. It is easy to see the inequality $\frac{(L^{2})-\sqrt{{(L^{2})}^{2}-4(p+2)^{2}(L^{2})}}{2(p+2)}< \frac{\sqrt{7(L^{2})}}{2\sqrt{2}}$, so we focus on the inequality $\frac{(L^{2})-\sqrt{{(L^{2})}^{2}-4(p+2)^{2}(L^{2})}}{2(p+2)}< \epsilon_{0}(L)$. Note that property $N_{p}$ for $(X,L)$ gives the inequality $\epsilon_{0}(L) \ge p+3$ by \cite[Theorem 4.1]{KL1509}. Hence, it remains to check the inequality $\frac{(L^{2})-\sqrt{{(L^{2})}^{2}-4(p+2)^{2}(L^{2})}}{2(p+2)}< p+3$. However, it comes from our assumption on $L^{2}>(p+2)(p+3)^{2}$, so we are done. 
\end{proof}

In a similar manner, we present a criterion for the higher order embeddings of $(X,L)$. (See \cite[Corollary 1.5]{KL1509} for the other numerical criterion.)  

\begin{cor} \label{cor:large p-very ampleness}
Let $(X,L)$ be a polarized abelian surface. Assume that $L^{2}>(p+1)(p+2)^{2}$. Then the following are equivalent:
\begin{enumerate}[(1)]
\item $\epsilon(X,L)>\frac{(L^{2})-\sqrt{{(L^{2})}^{2}-4(p+1)^{2}(L^{2})}}{2(p+1)}$.
\item $L$ is $p$-very ample. 
\end{enumerate}
\end{cor}

\begin{proof}
The implication $(1) \Rightarrow (2)$ follows from the inequality $\sqrt{L^{2}}<\frac{(L^{2})+\sqrt{{(L^{2})}^{2}-4(p+1)^{2}(L^{2})}}{2(p+1)}$ and Corollary \ref{cor:p-very ampleness}, so we need to check the reverse one. For $p \ge 1$, it is sufficient to show that $L^{2} \ge 5(p+1)^{2}$ by \cite[Theorem 1.1 and Corollary 1.5]{KL1509} and Theorem \ref{thm:large}. It is clear for $p \ge 2$, so we may let $p=1$. In this case, $L^{2}>18$. However, since $L^{2}$ is even, we have $L^{2} \ge 20=5(p+1)^{2}$ for $p=1$ as wanted. 

Thus we are left to check it for $p=0$. By the similar argument used in the proof of Theorem \ref{thm:large}, it is sufficient to show the following:
\begin{align*}
\frac{(L^{2})-\sqrt{{(L^{2})}^{2}-4(L^{2})}}{2}< \min\{\frac{\sqrt{7(L^{2})}}{2\sqrt{2}}, \text{ }\epsilon_{0}(L)\},
\end{align*}
where $\epsilon_{0}(L)$ is the minimal degree of an elliptic curve in $X$ with respect to $L$. Since the inequality $\frac{(L^{2})-\sqrt{{(L^{2})}^{2}-4(L^{2})}}{2}<\frac{\sqrt{7(L^{2})}}{2\sqrt{2}}$ is trivial, we concentrate on the second one. Since $L$ is globally generated, \cite[Theorem 1.1]{T} gives the non-existence of an elliptic curve $C$ on $X$ satisfying $(L.C)=1$, that is, $\epsilon_{0}(L) \ge 2$. Now, it is easy to see that the inequality $\frac{(L^{2})-\sqrt{{(L^{2})}^{2}-4(L^{2})}}{2}<2$ holds. Hence we are done. 
\end{proof}

\end{subsection}

\end{section}

\begin{section} {Applications} \label{section:application}
When applying Theorem \ref{thm:LPP}, it has some restriction: it only depends on the lower bound of the Seshadri constant. However, since Theorem \ref{thm:main} has two variables: the Seshadri constant and the self-intersection number, it seems likely to be more flexible than Theorem \ref{thm:LPP} when we treat property $N_{p}$ for a polarized abelian surface. The main point of this section is to extend Corollary \ref{cor:Seshadri LPP} and Corollary \ref{cor:Koszul LPP} on a polarized abelian surface by applying Theorem \ref{thm:main} and Proposition \ref{prop:Koszul} (cf. Corollary \ref{cor:general criterion} and Corollary \ref{cor:Koszul extension}). Furthermore, we analyze higher syzygies of line bundles on the self-product of an elliptic curve without complex multiplication. 

\begin{subsection} {Higher syzygies on abelian surfaces of type $(1,d)$ with Picard number one}
Let $(X,L)$ be a polarized abelian surface of type $(1,d)$ with Picard number one. Its Seshadri constant is computed as follows. 

\begin{lem} {\rm (\cite[Theorem 6.1]{T99})} \label{lem:Seshadri constant with one}
Let $(X,L)$ be as above. 
\begin{enumerate}[(1)]
\item If $\sqrt{2d}$ is rational, then $\epsilon(X,L)=\sqrt{2d}$. 
\item If $\sqrt{2d}$ is irrational, then $\epsilon(X,L)=2d \cdot \frac{k_{0}}{l_{0}}=\frac{2d}{\sqrt{2d+\frac{1}{k_{0}^{2}}}}$, where $(k_{0}, l_{0})$ is the primitive solution of the Pell equation $l^{2}-2dk^{2}=1$. 
\end{enumerate}
\end{lem}

By Lemma \ref{lem:Seshadri constant with one}, we find the lower bound of $d$ satisfying property $N_{p}$ for $(X,L)$ by Theorem \ref{thm:main} and Corollary \ref{cor:generalization LPP}. 

\begin{prop} \label{prop:Picard number one}
Let $(X,L)$ be a polarized abelian surface of type $(1,d)$ with Picard number one. If $d \ge 2(p+2)^{2}$, then $L$ satisfies property $N_{p}$. 
\end{prop}

\begin{proof}
First, consider the case when $\sqrt{2d}$ is rational. By Lemma \ref{lem:Seshadri constant with one}, $\epsilon(X,L)=\sqrt{2d}$. Since $d \ge 2(p+2)^{2}$, $\epsilon(X,L) \ge 2(p+2)$. By Corollary \ref{cor:generalization LPP}, $L$ satisfies property $N_{p}$. Next, assume that $\sqrt{2d}$ is irrational. Also, by Lemma \ref{lem:Seshadri constant with one}, $\epsilon(X,L)=\frac{2d}{\sqrt{2d+\frac{1}{k_{0}^{2}}}}$. Since $\sqrt{2d}$ is irrational, $d \ge 2(p+2)^{2}+1$. Then we have 
\begin{align*}
\epsilon(X,L)=\frac{2d}{\sqrt{2d+\frac{1}{k_{0}^{2}}}} \ge \frac{2d}{\sqrt{2d+1}} \ge \frac{2(2(p+2)^{2}+1)}{\sqrt{4(p+2)^{2}+3}}>2(p+2)
\end{align*}
since $\frac{2d}{\sqrt{2d+1}}$ has its minimum at $d=2(p+2)^{2}+1$. By Corollary \ref{cor:generalization LPP} or Theorem \ref{thm:LPP}, we are done. 
\end{proof}

\begin{rmk} 
Theorem \ref{thm:LPP} implies that $L$ satisfies property $N_{p}$ if $d>2(p+2)^{2}$ (cf. \cite[Corollary B]{LPP}). Comparing it with Proposition \ref{prop:Picard number one}, Theorem \ref{thm:main} does not look quite better than Theorem \ref{thm:LPP}. The main reason for this situation is because $\epsilon(X,L) \approx \sqrt{2d}=\sqrt{(L^{2})}$ so that the condition on Theorem \ref{thm:main} becomes similar to $\epsilon(X,L)>2(p+2)$, i.e. 
\begin{align*}
(L^{2}) \cdot (\epsilon(X,L)-p-2)-(p+2) \cdot \epsilon(X,L)^{2} \approx (L^{2}) \cdot (\epsilon(X,L)-2(p+2)).
\end{align*}
Thus Theorem \ref{thm:main} becomes more powerful than Theorem \ref{thm:LPP} when $\sqrt{(L^{2})}-\epsilon(X,L)$ is large. (See the next subsections.) 
\end{rmk}

\end{subsection}

\begin{subsection} {Higher syzygies on very general polarized abelian surfaces}
Let $(X,L)$ be a very general polarized abelian surface of type $(d_{1},d_{2})$ (with no restriction on its Picard number). In this case, Bauer (\cite{T97}) found the lower bound of $\epsilon(X,L)$ as follows.

\begin{lem} {\rm (\cite[Theorem 1.(b)]{T97})} \label{lem:lower bound}
With notation as above, 
\begin{align*}
\epsilon(X,L) \ge \frac{1}{2} \sqrt{d_{1}d_{2}}.
\end{align*}
\end{lem}

By Lemma \ref{lem:lower bound} and Theorem \ref{thm:main}, we obtain Corollary \ref{cor:general criterion}, which gives a nice lower bound of $L^{2}$ for higher syzygies of $(X,L)$. 

\begin{cor} \label{cor:general criterion}
Let $(X,L)$ be a polarized abelian surface, and let $p \ge 0$ be an integer. 
\begin{enumerate}[(1)]
\item Assume that $(X,L)$ is very general. If $L^{2}>\frac{81}{8}(p+2)^{2}$, then $L$ satisfies property $N_{p}$. In particular, the converse holds if $L^{2}>(p+2)(p+3)^{2}$. 
\item Assume that $\epsilon(X,L) \notin \mathbb{Z}$. If $L^{2}>\frac{9}{2}(p+2)^{2}$, then $L$ satisfies property $N_{p}$.
\end{enumerate}
\end{cor}

\begin{proof}
$(1)$ Let $(d_{1},d_{2})$ be a type of $L$. By Lemma \ref{lem:lower bound}, note that $\frac{1}{2}\sqrt{d_{1}d_{2}} \le \epsilon(X,L) \le \sqrt{2d_{1}d_{2}}$. Consider a function $\alpha(t)$ on $[\frac{1}{2}\sqrt{d_{1}d_{2}}, \sqrt{2d_{1}d_{2}}]$ defined by 
\begin{align*}
\alpha(t):&=(L^{2})(t-p-2)-(p+2)t^{2}=-(p+2)t^{2}+(L^{2})t-(p+2)(L^{2}).
\end{align*}
Clearly, we can consider $\alpha(t)$ as a function on $\mathbb{R}$ and in this case, it has its maximum value at $t=\frac{(L^{2})}{2(p+2)}=\frac{d_{1}d_{2}}{p+2}$. Since $d_{1}d_{2}>\frac{81}{16}(p+2)^{2}>2(p+2)^{2}$, $\sqrt{2d_{1}d_{2}}\le \frac{d_{1}d_{2}}{p+2}$. So since $\epsilon(X,L) \in [\frac{1}{2}\sqrt{d_{1}d_{2}}, \sqrt{2d_{1}d_{2}}]$, $\alpha(\epsilon(X,L)) \ge \alpha(\frac{1}{2}\sqrt{d_{1}d_{2}})$. However, also since $d_{1}d_{2}>\frac{81}{16}(p+2)^{2}$, 
\begin{align*}
\alpha(\frac{1}{2}\sqrt{d_{1}d_{2}})&=2d_{1}d_{2} \cdot (\frac{1}{2}\sqrt{d_{1}d_{2}}-p-2)-(p+2)\cdot \frac{1}{4}d_{1}d_{2} =d_{1}d_{2}(\sqrt{d_{1}d_{2}}-\frac{9}{4}(p+2))>0,
\end{align*}
i.e. $\alpha(\epsilon(X,L)) > 0$. By Theorem \ref{thm:main}, $L$ satisfies property $N_{p}$. The second statement in $(1)$ is a consequence of Theorem \ref{thm:large}.

$(2)$ Assume that $\epsilon(X,L) \notin \mathbb{Z}$ and $d_{1}d_{2}>\frac{9}{4}(p+2)^{2}$. Note that $\epsilon(X,L) \in \mathbb{Q}-\mathbb{Z}$ by \cite[Theorem 6.4.5]{TSBMAWT} and that $\epsilon(X,L)=\epsilon(L;x)$ for a very general point $x \in X$. We may assume that $\epsilon(X,L)$ is submaximal so that $\epsilon(X,L)=\frac{t}{q}$, where $t=(L.C)$ and $q={\rm mult}_{x}(C)$ for a Seshadri exceptional curve $C$ of $L$ on $x$. Since $\epsilon(X,L) \in \mathbb{Q}-\mathbb{Z}$, $q \ge 2$. Now, consider the generic infinitesimal Newton-Okounkov body $\Delta_{x}(L)$. By Proposition \ref{prop:Seshadri surface}, $\Delta_{x}(L) \subseteq \Delta_{ODR}$, where $O=(0,0)$, $D=(\frac{t}{q}, \frac{t}{q})$, and $R=(\frac{t}{q-1},0)$. Thus, $\frac{1}{2}(L^{2}) \le \frac{1}{2} \cdot \frac{t^{2}}{q(q-1)}$. Since $\frac{t^{2}}{q(q-1)}=\epsilon(X,L)^{2} \cdot \frac{q}{q-1}$, $\epsilon(X,L) \ge \sqrt{\frac{q-1}{q}(L^{2})}$. Since $\frac{q-1}{q}$ is an increasing function with $q \ge 2$, 
\begin{align*}
\epsilon(X,L) \ge \sqrt{\frac{q-1}{q}(L^{2})} \ge \sqrt{\frac{1}{2}(L^{2})},
\end{align*}
i.e. $\sqrt{\frac{1}{2}(L^{2})} \le \epsilon(X,L) \le \sqrt{(L^{2})}$. As in the proof of $\textit{(1)}$, since $d_{1}d_{2} \ge 2(p+2)^{2}$, $\alpha(t)$ is an increasing function on $[\sqrt{\frac{1}{2}(L^{2})}, \sqrt{(L^{2})}]=[\sqrt{d_{1}d_{2}}, \sqrt{2d_{1}d_{2}}]$, i.e. $\alpha(\epsilon(X,L)) \ge \alpha(\sqrt{d_{1}d_{2}})$. Since $d_{1}d_{2}>\frac{9}{4}(p+2)^{2}$, 
\begin{align*}
\alpha(\sqrt{d_{1}d_{2}})&=2d_{1}d_{2} \cdot (\sqrt{d_{1}d_{2}}-p-2)-(p+2)d_{1}d_{2} =d_{1}d_{2}(2 \sqrt{d_{1}d_{2}}-3(p+2))>0,
\end{align*}
i.e. $\alpha(\epsilon(X,L))>0$. By Theorem \ref{thm:main}, $L$ satisfies property $N_{p}$. 
\end{proof}

\begin{rmk} \label{rmk:comparison with LPP}
Note that Theorem \ref{thm:LPP} implies that property $N_{p}$ holds for $L$ if $L^{2}>32(p+2)^{2}$ (or see \cite[Corollary B]{LPP}). However, Corollary \ref{cor:general criterion} says that the same conclusion holds even when $L^{2}>\frac{81}{8}(p+2)^{2}$, which gives a better bound. Moreover, it is immediate that property $N_{p}$ holds not only for $L$ but also for $L+F$, where $F$ is any effective divisor on $X$, under these circumstances. 
\end{rmk}

The next application has a more classical flavor as it deals with the multiples of ample divisors (cf. \cite{P, PP}). The following result implies that the multiple needed for the higher syzygies of $L$ can be reduced as $L^{2}$ increases. 

\begin{cor} \label{cor:p+2}
Let $(X,L)$ be a polarized abelian surface, and let $p \ge 0$ be an integer. 
\begin{enumerate}[(1)]
\item Assume that $(X,L)$ is very general. Then $L^{\otimes \lceil \frac{9}{2\sqrt{2(L^{2})}} (p+2) \rceil}$ satisfies property $N_{p}$.
\item Assume that $\epsilon(X,L) \notin \mathbb{Z}$. Then $L^{\otimes \lceil \frac{3}{\sqrt{2(L^{2})}} (p+2) \rceil}$ satisfies property $N_{p}$. 
\end{enumerate}
\end{cor}

Finally, the Koszulness of $R(X,L)$ is obtained in the same manner.

\begin{cor} \label{cor:Koszul extension}
Let $(X,L)$ be a very general polarized abelian surface. If $(L^{2})>\frac{729}{8}$, then $R(X,L)$ is Koszul. 
\end{cor}

\end{subsection}

\begin{subsection}{Higher syzygies on $E \times E$ without complex multiplication}
Our last example is the self-product of an elliptic curve. We start by fixing notation. Denote by $E$ an elliptic curve without complex multiplication. We set $X=E \times E$ with projections $pr_{1}$, $pr_{2}:X \rightarrow E$. Fixing a point $P \in E$, consider the three classes $F_{1}=[\{P\} \times E]$, $F_{2}=[E \times \{P\}]$, and $\Delta$ in ${N^{1}(X)}_{\mathbb{R}}$, where $\Delta$ is the diagonal. It is well-known that they are linearly independent and span ${N^{1}(X)}_{\mathbb{R}}$. We recall the result of Bauer and Schulz (\cite{BS}):

\begin{thm} {\rm (\cite[Theorem 1]{BS})} \label{thm:Seshadri constant self-product}
Let $L=\OO_{X}(b_{1}F_{1}+b_{2}F_{2}+b_{3}\Delta)$	 be an ample line bundle on $X$, and take a permutation $(a_{1},a_{2},a_{3})$ of $(b_{1},b_{2},b_{3})$ satisfying $a_{1} \ge a_{2} \ge a_{3}$. 

Then $\epsilon(X,L)$ is the minimum of the following finitely many numbers:
\begin{enumerate}[(1)]
\item $a_{2}+a_{3}$,
\item $\frac{a_{2}a_{1}^{2}+a_{1}a_{2}^{2}+a_{3}(a_{1}+a_{2})^{2}}{{{\rm gcd}(a_{1},a_{2})}^{2}}$,
\item $\min\{a_{1}d^{2}+a_{2}c^{2}+a_{3}(c+d)^{2} \text{ }|\text{ } c,d \in \mathbb{N} \text{ coprime, } c+d<\frac{1}{\sqrt{2}}(a_{1}+a_{2})\}$.
\end{enumerate} 
\end{thm}

Now, we give explicit bounds of $a_{1}$, $a_{2}$, and $a_{3}$ for the higher syzygies of $L=\OO_{X}(b_{1}F_{1}+b_{2}F_{2}+b_{3}\Delta)$. 

\begin{cor} \label{cor:self-product}
Let $L=\OO_{X}(b_{1}F_{1}+b_{2}F_{2}+b_{3}\Delta)$	 be a line bundle on $X$, and take a permutation $(a_{1},a_{2},a_{3})$ of $(b_{1},b_{2},b_{3})$ satisfying $a_{1} \ge a_{2} \ge a_{3}$. For an integer $p \ge 0$, assume that either one of the following holds: 
\begin{enumerate}[(1)]
\item $a_{3} \ge 0$, $a_{2}+a_{3}>p+2$, and $a_{1}a_{2}+a_{2}a_{3}+a_{3}a_{1} > \frac{(p+2)(a_{2}+a_{3})^{2}}{2(a_{2}+a_{3}-p-2)}$, or
\item $-\sqrt{6}(p+2) \le a_{3}<0$, $a_{1}+3a_{3} \ge 0$, $a_{2}+a_{3}>p+2$, and $a_{1}a_{2}+a_{2}a_{3}+a_{3}a_{1}>\max\{\frac{(p+2)(a_{2}+a_{3})^{2}}{2(a_{2}+a_{3}-p-2)}, \frac{2(p+2){{\rm gcd}(a_{1},a_{2})}^{4}}{2(a_{1}+a_{2}){{\rm gcd}(a_{1},a_{2})}^{2}-(a_{1}+a_{2})^{2}(p+2)}\}$.
\end{enumerate}
Then $L$ satisfies property $N_{p}$. 
\end{cor}

\begin{proof}
It is standard and elementary (cf. \cite[(2.0.1)]{BS} or \cite[Lemma 4.3.2(b)]{BL}) that if our assumption $(1)$ or $(2)$ holds, then $L$ is ample. 

$(1)$ Since $a_{3} \ge 0$, $\epsilon(X,L)=a_{2}+a_{3}$ by \cite[Example 2.1]{BS}. By Theorem \ref{thm:main}, it suffices to prove 
\begin{align*}
a_{2}+a_{3}>\frac{a_{1}a_{2}+a_{2}a_{3}+a_{3}a_{1}-\sqrt{(a_{1}a_{2}+a_{2}a_{3}+a_{3}a_{1})^{2}-2(p+2)^{2}(a_{1}a_{2}+a_{2}a_{3}+a_{3}a_{1})}}{p+2}
\end{align*}
since $a_{1}a_{2}+a_{2}a_{3}+a_{3}a_{1} >\frac{(p+2)(a_{2}+a_{3})^{2}}{2(a_{2}+a_{3}-p-2)} \ge 2(p+2)^{2}$. We claim that $a_{1}a_{2}+a_{2}a_{3}+a_{3}a_{1}-(p+2)(a_{2}+a_{3}) \ge 0$. If $a_{1} \ge p+2$, it is obvious, so we assume that $a_{1}<p+2$. Then $a_{2}+a_{3}<2(p+2)$. Since $a_{1}a_{2}+a_{2}a_{3}+a_{3}a_{1}-(p+2)(a_{2}+a_{3})>2(p+2)^{2}-(p+2)(a_{2}+a_{3})=(p+2)(2(p+2)-a_{2}-a_{3})>0$, the claim holds. Thus it is enough to show the inequality:
\begin{align*}
(a_{1}a_{2}+a_{2}a_{3}+a_{3}a_{1})^{2}-2(p+2)^{2}(a_{1}a_{2}+a_{2}a_{3}+a_{3}a_{1}) > (a_{1}a_{2}+a_{2}a_{3}+a_{3}a_{1}-(p+2)(a_{2}+a_{3}))^{2}.
\end{align*} 
Now, the inequality follows from $a_{2}+a_{3}>p+2$ and $a_{1}a_{2}+a_{2}a_{3}+a_{3}a_{1} > \frac{(p+2)(a_{2}+a_{3})^{2}}{2(a_{2}+a_{3}-p-2)}$. 

$(2)$ We claim that
\begin{align*}
\min\{a_{1}d^{2}+a_{2}c^{2}+a_{3}(c+d)^{2} \text{ }|\text{ } c,d \in \mathbb{N} \text{ coprime, } c+d<\frac{1}{\sqrt{2}}(a_{1}+a_{2})\} \ge a_{2}+a_{3}.
\end{align*}
We need to show that $(a_{1}+a_{3})d^{2}+(a_{2}+a_{3})(c^{2}-1)+2a_{3}cd \ge 0$. If $c=1$, then the inequality $(a_{1}+a_{3})d^{2}+2a_{3}d \ge 0$ follows from $a_{1}+3a_{3} \ge 0$. So we may let $c \ge 2$. Let $f(d)=(a_{1}+a_{3})d^{2}+(2a_{3}c)d+(a_{2}+a_{3})(c^{2}-1)$ be a function on $d$. Now, it is sufficient to show that the discriminant of the quadratic function $f$ is non-positive, that is, $a_{3}^{2}c^{2}-(a_{1}+a_{3})(a_{2}+a_{3})(c^{2}-1) \le 0$. It follows from $a_{1}a_{2}+a_{2}a_{3}+a_{3}a_{1} >\frac{(p+2)(a_{2}+a_{3})^{2}}{2(a_{2}+a_{3}-p-2)} \ge 2(p+2)^{2}$ and $a_{3} \ge -\sqrt{6}(p+2)$;
\begin{align*}
a_{3}^{2}c^{2}-(a_{1}+a_{3})(a_{2}+a_{3})(c^{2}-1)&=a_{3}^{2}-(a_{1}a_{2}+a_{2}a_{3}+a_{3}a_{1})(c^{2}-1) \\
&<a_{3}^{2}-6(p+2)^{2} \le 0.
\end{align*}
So we have $\epsilon(X,L)=\min\{a_{2}+a_{3}, \text{ } \frac{(a_{1}+a_{2})(a_{1}a_{2}+a_{2}a_{3}+a_{3}a_{1})}{{{\rm gcd}(a_{1},a_{2})}^{2}}\}$. Then the rest of the arguments are similar to that on the proof of $(1)$, so we omit it here. 
\end{proof}

\end{subsection}

\end{section}

\bibliographystyle{abbrv}

\end{document}